\documentclass[12pt]{article}
\usepackage{graphicx} 			
\usepackage{amsfonts}
\usepackage{amsmath}
\usepackage{amssymb}
\usepackage{amsthm}
\usepackage{bbm}
\usepackage[paper=a4paper]{geometry}

\newtheorem{prop}{Proposition}[section]
\newtheorem{thm}[prop]{Theorem}

\newtheorem{lem}[prop]{Lemma}
\newtheorem{cor}[prop]{Corollary}
\newtheorem*{rem}{Remark}

\newcommand{\E}{\mathbf{E}}		
\newcommand{\V}{\mathbf{Var}}		
\newcommand{\Cov}{\mathbf{Cov}}		
\newcommand{\SM}{\mathcal{E}}		
\newcommand{\EM}{\mathcal{E}}		
\newcommand{\N}{\mathbf{N}}		
\newcommand{\R}{\mathbf{R}}		
\newcommand{\C}{\mathbf{C}}		
\newcommand{\F}{\mathcal{F}}		
\newcommand{\B}{\mathcal{B}}		
\newcommand{\est}[2]{\overline{#1}_{#2}}	
\newcommand{\X}{\mathbf{X}}		
\newcommand{\XS}{\mathcal{X}}		
\newcommand{\LT}{L^2(\pi)}		
\newcommand{\LTN}{L^2_0(\pi)}		
\renewcommand{\H}{\mathbf{H}}		
\newcommand{\norm}[1]{\left\lVert#1\right\rVert}	
\DeclareMathOperator{\range}{range}	
\DeclareMathOperator{\domain}{domain}	
\DeclareMathOperator{\nul}{null}	
\newcommand{\lft}{\langle}
\newcommand{\rgt}{\rangle}
\newcommand{\Vect}{\mathbf{V}}		
\renewcommand{\P}{\mathcal{P}}		
\newcommand{\Q}{\mathcal{Q}}		
\newcommand{\RO}{\mathcal{R}}		
\newcommand{\id}{\mathcal{I}}		
\newcommand{\I}{\mathbf{1}}		
\newcommand{\e}{\epsilon}		
\newcommand{\de}{\delta}		
\newcommand{\D}{\Delta}			
\renewcommand{\l}{\lambda}		
\newcommand{\g}{\gamma}			
\newcommand{\s}{\sigma}			
\renewcommand{\a}{\alpha}		
\newcommand{\one}{\mathbbm{1}}		

\title{Comparing the Efficiency of General State Space Reversible MCMC
Algorithms}
\author{}
\date{}

\begin{document}

\maketitle
\vspace{-1.5cm}
\begin{tabular}{ c c c }
	\large\text{Geoffrey T.\ Salmon}
	&\large\text{~~~~Jeffrey S.\ Rosenthal}\\
	\small\texttt{geoffrey.salmon@mail.utoronto.ca}
	&\small\texttt{~~~~jeff@math.toronto.edu}
\end{tabular}
\begin{center}
	\textit{Department of Statistical Sciences, University of Toronto}
\end{center}
\begin{center}
	\large\text{August 6, 2024}
\end{center}

\smallskip
\begin{abstract}
	We review and provide new proofs of results used to compare the efficiency of
	estimates generated by reversible Markov chain Monte Carlo
	(MCMC) algorithms on a general state space. We provide a full
	proof of the formula for the asymptotic variance for real-valued functionals
	on $\varphi$-irreducible reversible Markov chains, first introduced by Kipnis
	and Varadhan in
	\cite{kipnis varadhan}. Given two Markov kernels $P$ and $Q$ with stationary
	measure $\pi$, we say that the Markov kernel $P$ efficiency dominates
	the Markov kernel $Q$ if the asymptotic variance with respect to $P$ is at most
	the asymptotic variance with respect to $Q$ for every real-valued functional $f\in\LT$.
	Assuming only a basic background in functional analysis,
	we prove that for two $\varphi$-irreducible reversible Markov kernels $P$ and $Q$,
	$P$ efficiency dominates $Q$ if and only if the operator $\Q-\P$, where $\P$ is
	the operator on $\LT$ that maps $f\mapsto\int f(y)P(\cdot,dy)$ and similarly for
	$\Q$, is positive on $\LT$, i.e.\ $\lft f,\left(\Q-\P\right)f\rgt\geq0$ for every
	$f\in\LT$. We use this result to show that reversible antithetic kernels are more
	efficient than i.i.d.\ sampling, and that efficiency dominance is a partial
	ordering on $\varphi$-irreducible reversible Markov kernels. We also provide a
	proof based on that of Tierney in \cite{tierney} that Peskun dominance is
	a sufficient condition for efficiency dominance for reversible kernels.
	Using these results, we show that Markov kernels formed by randomly selecting
	other ``component'' Markov kernels will always efficiency dominate another Markov
	kernel formed in this way, as long as the component kernels of the former efficiency
	dominate those of the latter. These results on the efficiency dominance of
	combining component kernels generalises the
	results on the efficiency dominance of combined chains introduced by Neal and
	Rosenthal in \cite{finite} from finite state spaces to general state spaces.
\end{abstract}

\section{Introduction}\label{sect:intro}

A common problem in statistics and other areas,
is that of estimating the average value of a function $f:\X\to\R$
with respect to a probability measure $\pi$. The domain of $f$, $\X$, is called the
state space. When the probability measure
$\pi$ is complicated and the expected value of $f$ with respect to $\pi$, $\E_{\pi}(f)$,
can't be computed directly, Markov chain
Monte Carlo (MCMC) algorithms are very effective (see \cite{general}). In MCMC,
the solution is to find a Markov chain $\{X_k\}_{k\in\N}$ with
underlying Markov kernel $P$, and estimate the
expected value of $f$ with respect to $\pi$ as
$$\E_{\pi}(f)\approx\frac{1}{N}\sum_{k=1}^Nf(X_k)=:\est{f}{N}.$$
The advantage is that the Markov kernel $P$ provides much simpler probability measures
at each step, making it easier to compute, while the law of the Markov chain, $X_k$,
approaches the probability distribution $\pi$.

When the chosen function $f$ is in $\LT$, i.e.\ when
$\int_{\X}\left|f(x)\right|^2\pi(dx)\\=\E_{\pi}(|f|^2)<\infty$, one measure of the
effectiveness of the chosen Markov kernel for the function $f$, is the
\textit{asymptotic variance} of $f$ using the kernel $P$, $v(f,P)$, defined as
$$v(f,P):=
\lim_{N\to\infty}\left[N\V\left(\frac{1}{N}\sum_{k=1}^Nf(X_k)\right)\right]=
\lim_{N\to\infty}\left[\frac{1}{N}\V\left(\sum_{k=1}^Nf(X_k)\right)\right],$$
where $\{X_k\}_{k\in\N}$ is a Markov chain with kernel $P$, started in stationarity
(i.e.\ $X_1\sim\pi$).

Thus if $v(f,P)$ is finite, we would expect the variance of the estimate
$\est{f}{N}$ to be near $v(f,P)/N$. Furthermore, if $P$ is $\varphi$-irreducible and
reversible, a central limit theorem (CLT) holds whenever $v(f,P)$ is finite, and the
variance of said CLT is the asymptotic variance, i.e.\
$\sqrt{N}\left(\est{f}{N}-\E_{\pi}(f)\right)\overset{d}{\to}N(0,v(f,P))$ (see
\cite{kipnis varadhan}).
For further reference, see \cite{jones}, \cite{geyer}, \cite{new tierney}, \cite{general}
and \cite{kipnis varadhan}.

$X_1$ is not usually sampled from
$\pi$ directly, but if $P$ is $\varphi$-irreducible, then we can get very close to
sampling from $\pi$ directly by running the chain for a number of iterations before
using the samples for estimation.

In practice, it is common not to know in advance which function $f$ will be needed,
or need estimates for multiple functions. In these cases it is useful
to have a Markov chain to run estimates for various functions
simultaneously. Thus, we would like the variance of our estimates, and thus the
asymptotic variance, to be as low as possible for not just one function $f:\X\to\R$,
but as low as possible for every function $f:\X\to\R$ simultaneously.
This gives rise to the notion of an ordering of Markov kernels based on the
asymptotic variance of functions in $\LT$. Given two Markov kernels $P$ and $Q$ with
stationary distribution $\pi$, we
say that $P$ \textit{efficiency dominates} $Q$ if for every $f\in\LT$,
$v(f,P)\leq v(f,Q)$.

In this paper, we focus our attention on reversible Markov
kernels. Many important algorithms, most notably the Metropolis-Hastings algorithm,
are reversible (see \cite{tierney}, \cite{general}). 

Aside from Section \ref{sect:mixing}, many of the results of this paper are known but are
scattered in the literature, have incomplete or unclear proofs, or are missing proofs altogether.
We present
new clear, complete, and accessible proofs, using basic functional analysis where
very technical results were previously used, most notably in the proof of Theorem
\ref{equiv}. We show how once Theorem \ref{equiv} is established, many further
results are vastly simplified. This paper is self-contained assuming basic
Markov chain theory and functional analysis.

In Section \ref{sect:formula}, we provide a full proof of the formula for the
asymptotic variance of $\varphi$-irreducible reversible Markov kernels established by
Kipnis and Varadhan in \cite{kipnis varadhan},
$$v(f,P)=\int_{[-1,1)}\frac{1+\l}{1-\l}\,\EM_{f,\P}(d\l).$$
We also provide a full proof of a useful characterisation of the asymptotic variance
for aperiodic Markov kernels.

In Section \ref{sect:equivs}, we use the above formula as well as some
functional analysis from Section \ref{sect:lemma}, to show that efficiency dominance is
equivalent to a much simpler condition for $\varphi$-irreducible reversible kernels;
given Markov kernels $P$ and $Q$, for every $f\in\LTN$,
$\lft f,\P f\rgt\leq\lft f,\Q f\rgt$ (Theorem \ref{equiv}).
This equivalent condition is sometimes called
\textit{covariance dominance} (see \cite{mira geyer}).
The functional analysis used in the proof of Theorem \ref{equiv} is derived
from the basics in Section \ref{sect:lemma}.
We use this theorem to show that antithetic Markov kernels are more efficient than
i.i.d.\ sampling (Proposition \ref{antithetic}), and show that efficiency dominance
is a partial ordering (Theorem \ref{partial ordering}).

In Section \ref{sect:mixing}, we generalise the results on the efficiency dominance
of combined chains in \cite{finite} from finite state spaces to general state spaces.
Given reversible Markov kernels $P_1,\dots,P_l$ and $Q_1,\dots,Q_l$, we show that if
$\Q_k-\P_k$ is a positive operator on $\LTN$ for every $k$, and $\{\a_1,\dots,\a_l\}$ is a set
of mixing probabilities such that $P=\sum\a_kP_k$ and $Q=\sum\a_kQ_k$ are $\varphi$-irreducible,
then $P$ efficiency dominates $Q$ (Theorem \ref{mix}). This can be used to show that a
random-scan Gibbs sampler with more efficient component kernels will always be more efficient.
We also show that for two combined kernels differing in one component, one efficiency
dominates the other if and only if it's unique component kernel efficiency dominates the other's
(Corollary \ref{3 comb}).

In Section \ref{sect:peskun}, we consider Peskun dominance, or dominance off
the diagonal (see \cite{peskun}, \cite{tierney}). We say that a Markov kernel $P$
\textit{Peskun dominates} another kernel $Q$, if for $\pi$-a.e.\ $x\in\X$,
for every measurable set $A$,
$P(x,A-\{x\})\geq Q(x,A-\{x\})$. In Lemma \ref{pesk then pos}, we follow the
techniques of Tierney in \cite{tierney} to show that if $P$ Peskun dominates $Q$,
then $\Q-\P$ is a positive operator. We then show that with Theorem \ref{equiv}
established, the proof that Peskun dominance implies efficiency dominance is
simplified (Theorem \ref{peskun}).

\section{Background}\label{sect:prelims}

We are given the probability space $(\X,\F,\pi)$, where we assume the state
space $\X$ is non-empty.

\subsection{Markov Chain Background}\label{subsect:mc prelims}

A Markov kernel on $(\X,\F)$ with $\pi$ as a stationary
distribution is a function $P:\X\times\F\to[0,\infty)$ such that $P(x,\cdot)$
is a probability measure for every $x\in\X$, $P(\cdot,A)$ is a measurable
function for every $A\in\F$, and $\int_{\X}P(x,A)\pi(dx)=\pi(A)$ for every $A\in\F$.
The Markov kernel $P$ is \textit{reversible} with respect to $\pi$ if
$P(x,dy)\pi(dx)=P(y,dx)\pi(dy)$. A Markov kernel $P$ is \textit{$\varphi$-irreducible}
if there exists a non-zero $\s$-finite measure $\varphi$ on $(\X,\F)$ such that for
every $A\in\F$ with $\varphi(A)>0$, for every $x\in\X$ there exists $n\in\N$ such
that $P^n(x,A)>0$.

The space $\LT$ is defined rigorously as the set of
equivalence classes of $\pi$-square-integrable real-valued functions, with two functions
$f$ and $g$ being equivalent if $f=g$ $\pi$-a.e. Less rigorously,
$\LT$ is simply the set of $\pi$-square-integrable real-valued
functions. When this set is
endowed with the inner-product $\lft\cdot,\cdot\rgt:\LT\times\LT\to\R$ such
that $f\times g\mapsto\lft f,g\rgt:=\int_{\X}f(x)g(x)\pi(dx)$,
this space becomes a real Hilbert space. (When we are also dealing with complex functionals,
we define the inner-product instead to be
$f\times g\mapsto\lft f,g\rgt:=\int_{\X}f(x)\overline{g(x)}\pi(dx)$, where $\overline{\a}$
is the complex conjugate of $\a\in\C$, and $\LT$ becomes a complex Hilbert space. As we are
only dealing with real-valued functions, we do not need this distinction.)
Note that the norm induced by the
inner-product is the map $f\mapsto\norm{f}:=\sqrt{\lft f,f\rgt}$.

Recall from Section \ref{sect:intro} that for a function $f\in\LT$, it's \textit{asymptotic
variance} with respect to the Markov kernel $P$, denoted $v(f,P)$, is defined as
$v(f,P)\::=\:
\lim_{N\to\infty}\left[N\V\left(\frac{1}{N}\sum_{k=1}^Nf(X_k)\right)\right]\:=\:
\lim_{N\to\infty}\left[\frac{1}{N}\V\left(\sum_{k=1}^Nf(X_k)\right)\right]$, where\\
$\{X_k\}_{k\in\N}$ is a Markov chain with Markov kernel $P$ started in stationarity. Also from
Section \ref{sect:intro}, recall that given two Markov kernels $P$ and $Q$, both with
stationary measure $\pi$, $P$ \textit{efficiency dominates} $Q$ if
$v(f,P)\:\leq\:v(f,Q)$ for every $f\in\LT$.

For every Markov kernel $P$, we can define a linear operator $\P$ with
the space of $\F$ measurable functions as it's domain by
$$\P f(\cdot):=\int_{y\in\X}f(y)P(\cdot,dy).$$ For every Markov kernel, we denote
the associated linear operator defined above by it's letter in calligraphics.
If $P$ is stationary with respect to $\pi$, the range of $\P$ restricted to
$\LT$ is a subset of $\LT$, as for every $f\in\LT$, by Jensen's inequality
\begin{equation}\label{ineq:1}
	\int_{x\in\X}\left|\P f(x)\right|^2\pi(dx)\leq
	\iint_{x,y\in\X}\left|f(y)\right|^2P(x,dy)\pi(dx)=
	\int_{y\in\X}\left|f(y)\right|^2\pi(dy)<\infty.
\end{equation}
(see \cite{baxter}).
In this paper, we will only deal with functions in $\LT$, and thus view $\P$
as a map from $\LT\to\LT$, or a subset thereof.

Notice that the constant function, $\one:\X\to\R$ such that $\one(x)=1$ for
every $x\in\X$, is in $\LT$, and furthermore, as $P(x,\cdot)$ is a
probability measure for every $x\in\X$, $\P\one=\one$. Thus $\one$ is an eigenfunction
of $\P$ with eigenvalue $1$. We define the space $\LTN$ to be the subspace of
$\LT$ perpendicular to $\one$, i.e.\ $\LTN:=\{f\in\LT|f\perp\one\}=
\{f\in\LT|\lft f,\one\rgt=0\}=\{f\in\LT|\E_{\pi}(f)=0\}$. Notice that if $\P$
is restricted to $\LTN$ and $P$ is stationary with respect to $\pi$, then it's range
is contained in $\LTN$, as for every $f\in\LTN$, by Fubini's Theorem,
$\lft\P f,\one\rgt=\int_{y\in\X}f(y)\int_{x\in\X}P(x,dy)\pi(dx)=
\int_{\X}f(y)\pi(dy)=\lft f,\one\rgt=0$.

Furthermore, notice that if $P$ and $Q$ are Markov kernels with stationary measure
$\pi$, $P$ efficiency dominates $Q$ if and only if $P$ efficiency dominates $Q$ on
the smaller subspace $\LTN$. The forward implication is trivial as $\LTN\subseteq\LT$,
and for the converse, notice that for every $f\in\LT$, $v(f,P)=v(f_0,P)$ and $v(f,Q)=v(f_0,Q)$,
where $f_0:=f-\E_{\pi}(f)\in\LTN$.
Thus when talking about efficiency dominance, we can restrict ourselves to $\LTN$
instead of dealing with all of $\LT$, and we get rid of the eigenfunction $\one$.
Unless stated otherwise, we will consider $\P$ as an operator on and to $\LTN$.

A Markov kernel $P$ is \textit{periodic} with period $d\geq2$ if there exists
$\XS_1,\dots,\XS_d\in\F$ such that $\XS_k\cap\XS_j=\emptyset$ for every $j\neq k$,
and for every $i\in\{1,\dots,d-1\}$, $P(x,\XS_{i+1})=1$ for every
$x\in\XS_i$ and $P(x,\XS_1)=1$ for every $x\in\XS_d$. The sets $\XS_1,\dots,\XS_d\in
\F$ described above are called a periodic decomposition of $P$. A
Markov kernel $P$ is \textit{aperiodic} if it is not periodic.

A common definition related
to the efficiency of Markov kernels is the \textit{lag-k autocovariance}. For
an $\F$-measurable function $f$, the lag-k autocovariance, denoted $\g_k$, is the
covariance between $f(X_0)$ and $f(X_k)$, where $\{X_t\}_{t\in\N}$ is a Markov
chain run from stationarity with kernel $P$, i.e.\
$\g_k:=\Cov_{\pi,P}(f(X_0),f(X_k))$. When the function $f$ is in $\LTN$, notice
$\g_k=\E_{\pi,P}(f(X_0)f(X_k))=\lft f,\P^kf\rgt$.

We denote the Markov kernel associated with i.i.d.\ sampling from $\pi$ as $\Pi$,
i.e.\ $\Pi:\X\times\F\to[0,\infty)$ such that $\Pi(x,A)=\pi(A)$ for every $x\in\X$
and $A\in\F$. Notice that for every $f\in\LTN$, $\Pi f(x)=\E_{\pi}(f)=0$ for every
$x\in\X$. Thus $\Pi$ restricted to $\LTN$ is the zero function on $\LTN$.

\subsection{Functional Analysis Background}\label{subsect:fa prelims}

Here we present some functional analysis that will be used throughout the paper. For
a proper introduction to functional analysis, see Rudin's \cite{rudin}, or Conway's \cite{conway}.

An operator $T$ on a Hilbert space $\H$ is called \textit{bounded} if
there exists $C>0$ such that for every $f\in\H$, $\norm{Tf}\leq C\norm{f}$. The \textit{norm}
of a bounded operator is defined as the smallest such constant $C>0$ such that the above holds,
i.e.\ $\norm{T}:=\inf\{C>0:\norm{Tf}\leq C\norm{f},\quad \forall f\in\H\}$. A
bounded operator $T$ is called \textit{invertible} if it is bijective, and the
inverse of $T$, $T^{-1}$, is bounded. 

\textit{Unbounded} operators on $\H$ are linear operators $T$ such that there is no $C>0$
such that $\norm{Tf}\leq C\norm{f}$ for every $f\in\H$. Oftentimes, unbounded operators $T$
are not defined on the whole space $\H$, but only on a subset of $\H$. An unbounded operator
$T$ is \textit{densely defined} if the domain of $T$ is dense in $\H$.

The \textit{adjoint} of a bounded operator $T$
is the unique bounded operator $T^*$ such that $\lft Tf,g\rgt=\lft f,T^*g\rgt$ for
every $f,g\in\H$. Similarly, if $T$ is a densely defined operator, then the \textit{adjoint}
of $T$ is the linear operator $T^*$ such that $\lft Tf,g\rgt=\lft f,T^*g\rgt$ for every
$f\in\domain(T)$ and $g\in\H$ such that $f\mapsto\lft Tf,g\rgt$ is a bounded linear functional
on $\domain(T)$. Thus we define $\domain(T^*):=\{g\in\H|f\mapsto\lft Tf,g\rgt \text{ is a bounded
linear functional on }\domain(T)\}$. These two definitions are equivalent when $T$ is bounded.

A bounded operator $T$ is called \textit{normal}
if $T$ commutes with it's adjoint, $TT^*=T^*T$, and \textit{self-adjoint} if
$T$ equals it's adjoint, $T=T^*$. A densely defined operator $T$ is called \textit{self-adjoint}
if $T=T^*$ and $\domain(T)=\domain(T^*)$.

$\P$ restricted to $\LT$ is self-adjoint if and only if $P$ is reversible. As $\LTN
\subseteq\LT$, if $P$ is reversible with respect to $\pi$, then $\P$ restricted to $\LTN$
is self-adjoint.

The \textit{spectrum} of an operator $T$ is the set\\
$\s(T):=\{\l\in\C|T-\l\id\text{ is not invertible}\}$,
where $\id$ is the identity operator.
Note that invertible is meant in the context of bounded linear operators given
above. If the operator $T$ is self-adjoint, the spectrum of $T$ is real, i.e.\ $\s(T)
\subseteq\R$ (see Theorem 12.26 (a) in \cite{rudin}). It is important to note that
in the case the underlying
Hilbert space of the operator $T$ is finite-dimensional, as is the case for
$\LT$ and $\LTN$ when $\X$ is finite, that the spectrum of $T$ is exactly the set
of eigenvalues of $T$.

An operator $T$ on a Hilbert space $\H$ is called \textit{positive} if for every
$f\in\H$, $\lft f,Tf\rgt\geq0$. As we shall see in Lemma \ref{pos eigenvalues},
if $T$ is bounded and normal, then $T$ is positive if and only if the spectrum
of $T$ is positive, i.e.\ $\s(T)\subseteq[0,\infty)$. It is important to note
that when the Hilbert space $\H$ is a real Hilbert space, it is not necessarily
true that if $T$ is positive and bounded then $T$ is self-adjoint. This is an important
distinction, as this is true when $\H$ is a complex Hilbert space.

Furthermore, as shown by inequality (\ref{ineq:1}), (which is $\norm{\P f}^2\leq
\norm{f}^2$), the norm of $\P$ is less than or equal to $1$. This also bounds the
spectrum of $\P$ to $\l\in\C$ such that $|\l|\leq\norm{\P}\leq1$.
If $P$ is reversible, then $\P$ is self-adjoint and thus the spectrum of $\P$
is real, $\s(\P)\subseteq\R$. Thus $\s(\P)\subseteq[-1,1]$.

When $P$ is reversible, as $\P$ is self-adjoint, it is normal, and thus by the
Spectral Theorem (see Theorem 12.23 of \cite{rudin}, Chapter 9 Theorem 2.2 of \cite{conway}),
$\P$ has a spectral decomposition, $\P=\int_{\s(\P)}\l\SM_{\P}(d\l)$, where
$\SM_{\P}$ is the spectral measure of $\P$. For a bounded normal linear operator
$T:\H\to\H$ where $\H$ is a Hilbert space, for every $f\in\H$, we define the
induced measure $\EM_{f,T}$ as $\EM_{f,T}(A):=\lft f,\SM_{T}(A)f\rgt$ for every
Borel measurable set $A$ of $\C$. For every bounded Borel measureable function $\phi:\C\to\C$
and for
every bounded self adjoint operator $T$, define $\phi(T):=\int\phi(\l)\SM_{T}(d\l)$.
So, for any bounded Borel measureable function
$\phi:\C\to\C$, for every $f\in\H$,
$$\lft f,\phi(T)f\rgt=\int_{\s(T)}\phi(\l)\EM_{f,T}(d\l).$$

We will also usually assume that the Markov kernel $P$ is
$\varphi$-irreducible, as when $P$ is $\varphi$-irreducible,
the constant function is the only Eigenfunction (up to a scalar multiple)
of $\P$ with eigenvalue $1$ (see Lemma 4.7.4 of \cite{equivs}).
Thus if $P$ is $\varphi$-irreducible, then $1$ is not an eigenvalue of
$\P$ when restricted to $\LTN$. Note however that this does not mean that
$1\notin\s(\P)$ when restricted to $\LTN$.

\section{Asymptotic Variance}\label{sect:formula}

We now provide a detailed proof of the formula for the asymptotic variance of
$\varphi$-irreducible reversible Markov kernels, originally introduced by Kipnis
and Varadhan in \cite{kipnis varadhan}. Also in this section, we prove another
more familiar and practical characterisation of the asymptotic variance for
$\varphi$-irreducible reversible aperiodic Markov kernels (see \cite{green han},
\cite{geyer}, and \cite{jones}).

\begin{thm}\label{formula}
	If $P$ is a $\varphi$-irreducible reversible Markov kernel
	with stationary distribution $\pi$,
	then for every $f\in\LTN$,
	$$v(f,P)\:=\:\int_{\l\in[-1,1)}\frac{1+\l}{1-\l}\,\EM_{f,\P}(d\l),$$
	where $\EM_{f,\P}$ is the measure induced by the spectral measure of $\P$ (see Section
	\ref{subsect:fa prelims}). Note however, that this may still diverge to $\infty$.
\end{thm}
\begin{proof}
	For every $f\in\LTN$, by expanding the squares of
	$\V_{\pi,P}\left(\sum_{k=1}^Nf(X_k)\right)$, as $\E_{\pi}(f)=0$ (by
	definition of $\LTN$),
	$$\frac{1}{N}\V_{\pi,P}\left(\sum_{k=1}^Nf(X_k)\right)\,=\,
	\norm{f}^2+2\sum_{k=1}^N\left(\frac{N-k}{N}\right)\lft f,\P^kf\rgt.$$
	Thus as $\lft f,\P^kf\rgt\:=\:\int_{\s(\P)}\l^k\,\EM_{f,\P}(d\l)$ for every
	$k\in\N$,
	\begin{align*}
		v(f,P)
		&\,=\,\lim_{N\to\infty}\left[\frac{1}{N}
		\V\left(\sum_{n=1}^Nf(X_n)\right)\right]\\
		&\,=\,\lim_{N\to\infty}\left[\norm{f}^2+2\sum_{k=1}^N
		\left(\frac{N-k}{N}\right)\lft f,\P^kf\rgt\right]\\
		&\,=\,\norm{f}^2+2\lim_{N\to\infty}
		\left[\int_{\l\in\s(\P)}\sum_{k=1}^{\infty}
		\I_{k\leq N}(k)\left(\frac{N-k}{N}\right)
		\l^k\EM_{f,\P}(d\l)\right].\tag{1}\label{eq:1}
	\end{align*}

	In order to deal with the limit in (\ref{eq:1}), we split the integral
	over $\s(\P)$ into three subsets, $(0,1)$, $(-1,0]$ and $\{-1\}$. (Recall
	from Section \ref{sect:prelims} that $\s(\P)\subseteq[-1,1]$, and notice
	that we do not need to worry about $\{1\}$, as $1$ is not an eigenvalue
	of $\P$ on $\LTN$ as $P$ is $\varphi$-irreducible (see Section
	\ref{subsect:mc prelims}), and thus $\EM_{f,\P}(\{1\})=0$ by
	Lemma \ref{E(1)=0} below).

	For the first two subsets, for every fixed $\l\in(-1,0]\cup(0,1)=(-1,1)$,
	notice that for every $N\in\N$,
	$$\sum_{k=1}^{\infty}\left|\I_{k\leq N}(k)
	\left(\frac{N-k}{N}\right)\l^k\right|
	\leq\sum_{k=1}^{\infty}\left|\l^k\right|=\frac{|\l|}{1-|\l|}<\infty,$$
	so the sum is absolutely summable. Thus we can pull the pointwise
	limit through and show that for every $\l\in(-1,1)$, by the geometric
	series $\sum_{k=1}^{\infty}r^k=\frac{r}{1-r}$,
	$\lim_{N\to\infty}\left[\sum_{k=1}^{\infty}\I_{k\leq N}(k)
	\left(\frac{N-k}{N}\right)\l^k\right]
	=\frac{\l}{1-\l}$.

	By the Monotone Convergence Theorem for $\l
	\in(0,1)$ and the Dominated Convergence Theorem for $\l\in(-1,0]$,\\
	$\lim_{N\to\infty}\left[\int_{\l\in(0,1)}\sum_{k=1}^{\infty}
	\I_{k\leq N}(k)\left(\frac{N-k}{N}\right)\l^k\EM_{f,\P}(d\l)\right]
	\,=\,\int_{\l\in(0,1)}\frac{\l}{1-\l}\EM_{f,\P}(d\l),$ \:\:\:and\\
	$\lim_{N\to\infty}\left[\int_{\l\in(-1,0]}\sum_{k=1}^{\infty}
	\I_{k\leq N}(k)\left(\frac{N-k}{N}\right)\l^k\EM_{f,\P}(d\l)\right]
	\,=\,\int_{\l\in(-1,0]}\frac{\l}{1-\l}\EM_{f,\P}(d\l).$
	
	For the last case, the case of $\{-1\}$, notice that for every
	$N\in\N$, (simplifying the equation found in \cite{finite}), denoting the
	floor of $x\in\R$ as $\lfloor x\rfloor$,
	\begin{align*}
		&\sum_{k=1}^{\infty}\I_{k\leq N}(k)\frac{N-k}{N}(-1)^k\EM_{f,\P}
		(\{-1\})
		\:=\:\EM_{f,\P}(\{-1\})N^{-1}\sum_{k=1}^{N}\left[(-1)^k(N-k)\right]\\
		&\:=\:\EM_{f,\P}(\{-1\})N^{-1}\sum_{m=1}^{\lfloor N/2\rfloor}\left[(N-2m)
			-(N-2m+1)\right]\\
		&\:=\:\EM_{f,\P}(\{-1\})N^{-1}\sum_{m=1}^{\lfloor N/2\rfloor}\left(-1
			\right)
		\:=\:\left(\frac{\lfloor-N/2\rfloor}{N}\right)\EM_{f,\P}(\{-1\}).
	\end{align*}

	Thus as
	$\lim_{N\to\infty}\frac{\lfloor-N/2\rfloor}{N}=-1/2$, we have the pointwise
	limit\\
	$\lim_{N\to\infty}\sum
	\I_{k\leq N}(k) \left(\frac{N-k}{N}\right)\EM_{f,\P}(\{-1\})
	=\left(\frac{-1}{2}\right)\EM_{f,\P}(\{-1\})
	=\left(\frac{\l}{1-\l}\right)\EM_{f,\P}(\{\l\})|_{\l=-1}$.
	
	In order to split the integral in (\ref{eq:1}) into our three pieces,
	$(0,1)$, $(-1,0]$ and $\{-1\}$, and pull the limit through, we have to
	verify that if we do pull the limit through, we are not performing
	$\infty-\infty$. To verify this, it is enough to show that
	$\left|\lim_{N\to\infty}\int_{\l\in(-1,0]}\sum_{k=1}^{\infty}\I_{k\leq N}(k)
	\left(\frac{N-k}{N}\right)\l^k\EM_{f,\P}(d\l)\right|$ and\\
	$\left|\lim_{N\to\infty}\sum_{k=1}^{\infty}\I_{k\leq N}(k)
	\left(\frac{N-k}{N}\right)\EM_{f,\P}(\{-1\})\right|$ are finite. So from
	our work above, we find that\\
	$\left|\lim_{N\to\infty}\int_{\l\in(-1,0]}\sum_{k=1}^{\infty}\I_{k\leq N}(k)
		\left(\frac{N-k}{N}\right)\l^k\EM_{f,\P}(d\l)\right|
	=\left|\int_{\l\in(-1,0]}\frac{\l}{1-\l}\EM_{f,\P}(d\l)\right|<\infty$\\
	and
	$\left|\lim_{N\to\infty}\sum_{k=1}^{\infty}\I_{k\leq N}(k)
		\left(\frac{N-k}{N}\right)\EM_{f,\P}(\{-1\})\right|
	=\left|\left(\frac{-1}{2}\right)\EM_{f,\P}(\{-1\})\right|
	<\infty$.

	So, denoting $H(N,k):=\I_{k\leq N}(k)\left(\frac{N-k}{N}\right)$,
	\begin{align*}
		&\lim_{N\to\infty}\left[\int_{\l\in[-1,1)}\sum_{k=1}^{\infty}
		\I_{k\leq N}(k)\left(\frac{N-k}{N}\right)\l^k\EM_{f,\P}(d\l)\right]\\
		&=\lim_{N\to\infty}\left[\sum_{k=1}^{\infty}H(N,k)(-1)^k\EM_{f,\P}
		(\{-1\})\right.
		+\int_{\l\in(-1,0]}\sum_{k=1}^{\infty}H(N,k)\l^k\EM_{f,\P}(d\l)\\
		&\qquad\left.+\int_{\l\in(0,1)}\sum_{k=1}^{\infty}H(N,k)\l^k
		\EM_{f,\P}(d\l)\right]\\
		&=\lim_{N\to\infty}\sum_{k=1}^{\infty}H(N,k)(-1)^k\EM_{f,\P}(\{-1\})
		+\lim_{N\to\infty}\int_{\l\in(-1,0]}\sum_{k=1}^{\infty}
		H(N,k)\l^k\EM_{f,\P}(d\l)\\
		&\qquad+\lim_{N\to\infty}\int_{\l\in(0,1)}\sum_{k=1}^{\infty}
		H(N,k)\l^k\EM_{f,\P}(d\l)\\
		&=\left(\frac{\l}{1-\l}\right)\EM_{f,\P}(\{\l\})|_{\l=-1}
		+\int_{\l\in(-1,0]}\frac{\l}{1-\l}\EM_{f,\P}(d\l)
		+\int_{\l\in(0,1)}\frac{\l}{1-\l}\EM_{f,\P}(d\l)\\
		&=\int_{\l\in[-1,1)}\frac{\l}{1-\l}\EM_{f,\P}(d\l).
	\end{align*}
	Plugging this into (\ref{eq:1}), and as $\EM_{f,\P}(\{1\})=0$ by Lemma
	\ref{E(1)=0} below, we have
	\begin{align*}
		v(f,P)
		&=\norm{f}^2+2\lim_{N\to\infty}\left[\int_{\l\in\s(\P)}
		\sum_{k=1}^{\infty}
		\I_{k\leq N}(k)\left(\frac{N-k}{N}\right)\l^k\EM_{f,\P}(d\l)\right]\\
		&=\int_{\l\in[-1,1)}\EM_{f,\P}(d\l)+2\int_{\l\in[-1,1)}
		\frac{\l}{1-\l}\EM_{f,\P}(d\l)\\
		&=\int_{\l\in[-1,1)}\frac{1+\l}{1-\l}\EM_{f,\P}(d\l).
	\end{align*}
\end{proof}

\begin{lem}\label{E(1)=0} If $P$ is a $\varphi$-irreducible
	Markov kernel reversible with respect to $\pi$, then for every
	$f\in\LTN$, $\EM_{f,\P}(\{1\})=0$.
\end{lem}
\begin{proof} As outlined in \cite{equivs} Lemma 4.7.4,
	as $P$ is $\varphi$-irreducible, the constant function
	is the only eigenfunction of $\P$ with eigenvalue $1$. Thus $1$ is not
	an eigenvalue of $\P$ when restricted to $\LTN$.

	As seen in Theorem 12.29 (b) of \cite{rudin},
	for every normal bounded operator $T$ on a Hilbert space, if $\l\in\C$ is not
	an eigenvalue of $T$, then $\SM_T(\{\l\})=0$. As $P$ is
	reversible with respect to $\pi$, $\P$ is self-adjoint and thus also normal.

	So applying the above theorem to $\P$, as $1\in\C$ is not an eigenvalue
	of $\P$, $\SM_{\P}(\{1\})=0$. Thus
	for every $f\in\LTN$, $\EM_{f,\P}(\{1\})=\lft f,\SM_{\P}(\{1\})f\rgt=
	\lft f,0\rgt=0$.
\end{proof}

We now show that if $P$ is aperiodic, then $-1$ cannot be an eigenvalue of
$\P$.

\begin{prop}\label{aperiodicity -1} Let $P$ be a Markov kernel reversible
	with respect to $\pi$. If $P$ is aperiodic then $-1$ is not an eigenvalue of
	$\P:\LTN\to\LTN$.
\end{prop}
\begin{proof}
	We show the contrapositive. That is, we assume that $-1$ is an eigenvalue of $\P$,
	and show that $P$ is not aperiodic, i.e.\ periodic.

	Let $f\in\LTN$
	be an eigenfunction of $\P$ with eigenvalue $-1$. As $\P$ is self-adjoint
	(as $P$ is reversible),
	we can assume that $f$ is real-valued. Let
	$\XS_1=\{x\in\X:f(x)>0\}=f^{-1}((0,\infty))$ and
	$\XS_2=\{x\in\X:f(x)<0\}=f^{-1}((-\infty,0))$. As $f$ is
	$(\F,\B(\R))$-measureable where $\B(\R)$ is the Borel $\s$-field on $\R$,
	$\XS_1,\XS_2\in\F$.

	As $f$ is an eigenfunction, $f\neq0$ $\pi$-a.e. As $f\in\LTN$,
	$$0=\E_{\pi}(f)=\int_{\X}f(x)\pi(dx)=
	\int_{\XS_1}f(x)\pi(dx)+\int_{\XS_2}f(x)\pi(dx).$$
	So, the above combined with the fact that $f\neq0$ $\pi$-a.e.\ gives us that
	$\pi(\XS_1),\pi(\XS_2)>0$.

	So, as $\P f=-f$ for $\pi$-a.e.\ $x\in\X$ and
	$P$ is reversible with respect to $\pi$,
	$$\int_{\XS_1}f(x)\pi(dx)=-\int_{\XS_2}f(x)\pi(dx)=
	\int_{\XS_2}\P f(x)\pi(dx)=\int_{y\in\X}f(y)P(y,\XS_2)\pi(dy).$$
	Similarly, $\int_{\X}f(x)P(x,\XS_1)\pi(dx)=\int_{\XS_2}f(x)\pi(dx)$.

	We now claim that $P$ is periodic with respect to $\XS_1$
	and $\XS_2$. So assume for a contradiction there exists $E\in\F$ such that
	$\pi(E)>0$, $E\subseteq\XS_1$ and for every $x\in E$, $P(x,\XS_2)<1$.
	Then by definition of $\XS_2$,\\
	$\int_{\X}f(x)P(x,\XS_2)\pi(dx)\:=\:\int_{\XS_1}f(x)\pi(dx)
	\:>\:\int_{\XS_1}f(x)P(x,\XS_2)\pi(dx)\\
	\indent\geq\:\int_{\X}f(x)P(x,\XS_2)\pi(dx)$.\\
	So for $\pi$-a.e.\ $x\in\XS_1$, $P(x,\XS_2)=1$.
	Similarly, for $\pi$-a.e.\ $x\in\XS_2$, $P(x,\XS_1)=1$.

	Thus $P$ is periodic with period $d\geq2$.
\end{proof}

This proposition, combined with Theorem \ref{formula}, gives us a characterization
of $v(f,P)$ as sums of autocovariances, $\g_k$ (recall from Section \ref{subsect:mc prelims}),
when $P$ is aperiodic (see \cite{geyer}, \cite{green han}, \cite{jones}). Though this
characterization will not be used in this paper, it is perhaps more common and easier to
interpret from a statistical point of view.

\begin{prop}\label{autocovariances} If $P$ is an aperiodic $\varphi$-irreducible
	Markov kernel reversible with respect to $\pi$, then for every $f\in\LTN$,
	$$v(f,P)\:=\:\norm{f}^2+2\sum_{k=1}^{\infty}\lft f,\P^kf\rgt
	\:=\:\g_0+2\sum_{k=1}^{\infty}\g_k.$$
\end{prop}
\begin{rem} Even if $P$ is periodic, Proposition \ref{autocovariances} is still true
	for all $f\in\LTN$ that are perpendicular to the eigenfunctions of $\P$ with
	eigenvalue $-1$. (This ensures $\EM_{f,\P}(\{-1\})=0$).
\end{rem}
\begin{proof}
	Let $f\in\LTN$. By Proposition \ref{aperiodicity -1}, $-1$ is not an
	eigenvalue of $\P$. So, just as in the proof of Lemma \ref{E(1)=0}, again by
	Theorem 12.29 (b) of \cite{rudin}, as $\P$ is self-adjoint and thus normal,
	$\EM_{f,\P}(\{-1\})=0$.
	So by Theorem \ref{formula}, $v(f,P)=\int_{\l\in[-1,1)}\frac{1+\l}{1-\l}
	\EM_{f,\P}(d\l)=\norm{f}^2+2\int_{\l\in(-1,1)}\frac{\l}{1-\l}\EM_{f,\P}(d\l)$.
	
	Recalling the geometric series $\sum_{k=1}^{\infty}\l^k=\frac{\l}{\l-\l}$ for
	$\l\in(-1,1)$, by the Monotone Convergence Theorem for $\l\in(0,1)$ and by the
	Dominated Convergenve Theorem for $\l\in(-1,0]$, we have
	$$\int_{\l\in(-1,1)}\frac{\l}{1-\l}\EM_{f,\P}(d\l)
	\,=\,\int_{\l\in(-1,1)}\sum_{k=1}^{\infty}\l^k\EM_{f,\P}(d\l)
	\,=\,\sum_{k=1}^{\infty}\int_{\l\in(-1,1)}\l^k\EM_{f,\P}(d\l).$$

	So the asymptotic variance becomes $v(f,P)=\norm{f}^2+2\int_{\l\in(-1,1)}
	\frac{\l}{1-\l}\EM_{f,\P}(d\l)=\norm{f}^2+2\sum_{k=1}^{\infty}
	\int_{\l\in(-1,1)}\l^k\EM_{f,\P}(d\l)=\norm{f}^2+2\sum_{k=1}^{\infty}
	\lft f,\P^kf\rgt=\g_0+2\sum_{k=1}^{\infty}\g_k$, as $\E_{\pi}(f)=0$.
\end{proof}

\section{Efficiency Dominance Equivalence}\label{sect:equivs}

In this section, we use some basic functional analysis to prove our most useful
equivalent condition for efficiency dominance for $\varphi$-irreducible
reversible Markov kernels, simplifying the proof and staying away from overly technical
arguments. We then use this equivalent condition to show how
reversible antithetic Markov kernels are more efficient than i.i.d.\ sampling, and
show that efficiency dominance is a partial ordering on $\varphi$-irreducible
reversible kernels.

We state the equivalent condition theorem
here, and then introduce the lemmas we need from functional analysis before we
prove each direction of the equivalence. We prove the lemmas in Section
\ref{sect:lemma}.

\begin{thm}\label{equiv} If $P$ and $Q$ are $\varphi$-irreducible
	Markov kernels reversible with respect to $\pi$,
	then $P$ efficiency dominates $Q$ if and only if for every
	$f\in\LTN$,
	$$\lft f,\P f\rgt\leq\lft f,\Q f\rgt.$$
\end{thm}

To prove the ``if'' direction of Theorem \ref{equiv}, we follow an approach
similar to the one seen in \cite{mira geyer}, but with some notable differences.
The most important difference, is the use of the
following lemma instead of some overly technical results from \cite{bendat sherman}.
The following lemma, Lemma \ref{lemma 6}, is a generalisation of some results
found in Chapter V of \cite{bhatia} from finite dimensional vector spaces to general
Hilbert spaces. The finite dimensional version of Lemma \ref{lemma 6} is also
presented in \cite{finite} as Lemma 24.

\begin{lem}\label{lemma 6} If $T$ and $N$ are
	self-adjoint bounded linear operators on a Hilbert
	space $\H$, such that $\s(T),\s(N)\subseteq(0,\infty)$,
	then $\lft f,Tf\rgt\leq\lft f,Nf\rgt$ for every $f\in\H$,
	if and only if $\lft f,T^{-1}f\rgt\geq\lft f,N^{-1}f\rgt$,
	for ever $f\in\H$.
\end{lem}

Lemma \ref{lemma 6} is proven in Section \ref{subsect:lemma 6}, where we discuss
the differences in Lemma \ref{lemma 6} between finite dimensional
(as shown in \cite{finite})
and general Hilbert spaces.

We now prove the ``if'' direction of Theorem \ref{equiv} with the help of
Lemma \ref{lemma 6}.

\bigskip
\noindent
\textit{Proof of ``if'' Direction of Theorem \ref{equiv}.}

	Say $\lft f,\P f\rgt\leq\lft f,\Q f\rgt$ for every $f\in\LTN$.
	For every $\eta\in[0,1)$, let $T_{\P,\eta}=\id-\eta\P$ and
	$T_{\Q,\eta}=\id-\eta\Q$. Then as $\norm{\P},\norm{\Q}\leq1$,
	by the Cauchy-Schwartz inequality, for every $f\in\LTN$, 
	$\left|\lft f,\P f\rgt\right|\leq\norm{f}^2$ and
	$\left|\lft f,\Q f\rgt\right|\leq\norm{f}^2$. So again by
	the Cauchy-Schwartz inequality, for every $f\in\LTN$,
	\begin{align*}
		\norm{T_{\P,\eta}f}\norm{f}
		&\geq\left|\lft T_{\P,\eta}f,f\rgt\right|\\
		&=\left|\norm{f}^2-\eta\lft f,\P f\rgt\right|\\
		&\geq\left|1-\eta\right|\norm{f}^2.
	\end{align*}
	Thus for every $f\in\LTN$, $\norm{T_{\P,\eta}f},\norm{T_{\Q,\eta}f}\geq\left|1-\eta\right|
	\norm{f}$. As $\eta \in[0,1)$, $\left|1-\eta\right|>0$, and as $T_{\P,\eta}$
	and $T_{\Q,\eta}$ are both normal (as they are self-adjoint)
	$T_{\P,\eta}$ and $T_{\Q,\eta}$ are both invertible, in the sense
	of bounded linear operators, i.e.\ the inverse of $T_{\P,\eta}$ and
	$T_{\Q,\eta}$ are bounded, and
	$0\notin\s(T_{\P,\eta}),\s(T_{\Q,\eta})$ (by Lemma \ref{delt}).

	As $\norm{\P},\norm{\Q}\leq1$ and $\P$ and $\Q$ are self-adjoint (as
	$P$ and $Q$ are reversible),
	$\s(\P),\s(\Q)\subseteq[-1,1]$. Thus for every
	$\eta\in[0,1)$, $\s(T_{\P,\eta}),\s(T_{\Q,\eta})\subseteq(0,2)
	\subseteq(0,\infty)$.

	So for every $\eta\in[0,1)$, as $T_{\P,\eta}$ and $T_{\Q,\eta}$
	are both self-adjoint, and for every $f\in\LTN$,
	$\lft f,T_{\Q,\eta}f\rgt=\norm{f}^2-\eta\lft f,\Q f\rgt
	\leq\norm{f}^2-\eta\lft f,\P f\rgt=\lft f,T_{\P,\eta}f\rgt$,
	by Lemma \ref{lemma 6}, for every $f\in\LTN$,
	$\lft f,T^{-1}_{\Q,\eta}f\rgt\geq\lft f,T^{-1}_{\P,\eta}f\rgt$.

	Notice that for every $\eta\in[0,1)$,
	$T^{-1}_{\P,\eta}=\left(\id-\eta\P\right)\left(\id-\eta\P\right)^{-1}
	+\eta\P\left(\id-\eta\P\right)^{-1}=\id+\eta\P\left(\id-\eta\P\right)^{-1}$.
	So, for every $f\in\LTN$,
	\begin{align*}
		\norm{f}^2+\eta\lft f,\P\left(\id-\eta\P\right)^{-1}f\rgt
		&=\lft f,T^{-1}_{\P,\eta}f\rgt\\
		&\leq\lft f,T^{-1}_{\Q,\eta}f\rgt
		=\norm{f}^2+\eta\lft f,\Q\left(\id-\eta\Q\right)^{-1}f\rgt,
	\end{align*}
	so for every $f\in\LTN$, $\lft f,\P\left(\id-\eta\P\right)^{-1}f\rgt\leq
	\lft f,\Q\left(\id-\eta\Q\right)^{-1}f\rgt$.

	Thus by the Monotone Convergence Theorem for $\l\in(0,1)$ and the
	Dominated Convergence Theorem for $\l\in[-1,0]$, for every $f\in\LTN$,\\
	$\lim_{\eta\to1^-}\int_{[-1,1)}\frac{\l}{1-\eta\l}\EM_{f,\P}(d\l)=
	\int_{[-1,1)}\frac{\l}{1-\l}\EM_{f,\P}(d\l)$, and similarly for $\Q$.

	As $P$ and $Q$ are $\varphi$-irreducible and reversible with respect
	to $\pi$, by Theorem \ref{formula}, for every $f\in\LTN$,
	\begin{align*}
		&v(f,P)
		\;=\;\int_{[-1,1)}\frac{1+\l}{1-\l}\EM_{f,\P}(d\l)
		\;=\;\norm{f}^2+2\int_{[-1,1)}\frac{\l}{1-\l}\EM_{f,\P}(d\l)\\
		&\;=\;\norm{f}^2+2\lim_{\eta\to1^-}\int_{[-1,1)}\frac{\l}{1-\eta\l}
			\EM_{f,\P}(d\l)
		\;=\;\norm{f}^2+2\lim_{\eta\to1^-}\lft f,\P\left(\id-\eta\P\right)^{-1}f\rgt\\
		&\;\leq\;\norm{f}^2+2\lim_{\eta\to1^-}\lft f,\Q\left(\id-\eta\Q\right)^{-1}f\rgt
		\;=\;\norm{f}^2+2\lim_{\eta\to1^-}\int_{[-1,1)}\frac{\l}{1-\eta\l}
			\EM_{f,\Q}(d\l)\\
		&\;=\;\norm{f}^2+2\int_{[-1,1)}\frac{\l}{1-\l}\EM_{f,\Q}(d\l)
		\;=\;\int_{[-1,1)}\frac{1+\l}{1-\l}\EM_{f,\Q}(d\l)
		\;=\;v(f,Q).
	\end{align*}
	So as $v(f,P)\leq v(f,Q)$ for every $f\in\LTN$, $v(f,P)\leq v(f,Q)$ for every
	$f\in\LT$ (see Section \ref{subsect:mc prelims}), thus $P$ efficiency
	dominates $Q$.
\qed

\begin{rem} For the other direction of Theorem \ref{equiv}, it is hard to make use of
	any arguments utilising Lemma \ref{lemma 6}. If $P$ effieciency dominates $Q$,
	we are given that each \emph{individual} $f\in\LTN$ satisfies
	$\lim_{\eta\to1^-}\lft f,T_{\P,\eta}^{-1}f\rgt\leq
	\lim_{\eta\to1^-}\lft f,T_{\Q,\eta}^{-1}f\rgt$. As we can't apply Lemma
	\ref{lemma 6} to the limits above, it seems the most we can do is fix an $\e>0$,
	and by the above limit for every $f\in\LTN$ there exists $\eta_f\in[0,1)$ such
	that for every $\eta_f\leq\eta<1$, $\lft f,T_{\P,\eta}^{-1}f\rgt\leq
	\lft f,T_{\Q,\eta}^{-1}f\rgt+\e\norm{f}^2=\lft f,\left(T^{-1}_{\Q,\eta_f}+\e\id\right)
	f\rgt$. However, this results in a possible different $\eta_f$ for every $f\in\LTN$,
	and it's not obvious that $\sup\{\eta_f:f\in\LTN\}<1$, leaving us with no
	single $\eta\in[0,1)$ such that the above inequality holds for every
	$f\in\LTN$ to allow us to apply Lemma \ref{lemma 6}.
\end{rem}

Due to the difficulties in applying Lemma \ref{lemma 6} in the only if direction
of Theorem \ref{equiv}, we make use of the following lemma, which appears as
Corollary 3.1 from \cite{mira geyer}.

\begin{lem}\label{lemma 6.5} Let $T$ and $N$ be injective self-adjoint positive
	bounded linear operators on the Hilbert space $\H$ (though $T^{-1/2}$ and
	$N^{-1/2}$ may be unbounded).
	If $\domain(T^{-1/2})\subseteq\domain(N^{-1/2})$
	\emph{and} for every $f\in\domain(T^{-1/2})$ we have
	$\norm{N^{-1/2}f}\leq\norm{T^{-1/2}f}$,
	then $\lft f,Tf\rgt\leq\lft f,Nf\rgt$ for every $f\in\H$.
\end{lem}
See Section \ref{subsect:lemma 6.5} for the proof of Lemma \ref{lemma 6.5}.

We must also make use of the fact that the space of functions with finite asymptotic
variance is the domain of $\left(\id-\P\right)^{-1/2}$. This was first stated in
\cite{kipnis varadhan}, using a test function argument. We take a different approach
and provide a proof using the ideas of the Spectral Theorem. In particular, it
uses some ideas from the proof of Theorem X.4.7 from \cite{conway}.

\begin{lem}\label{space of fin var} If $P$ is a $\varphi$-irreducible
	Markov kernel reversible with respect to $\pi$, then
	$$\left\{f\in\LTN:v(f,P)<\infty\right\}=
	\domain\left(\left(\id-\P\right)^{-1/2}\right).$$
\end{lem}
\begin{proof}
	Let $\phi:[-1,1]\to\R$ such that $\phi(\l)=(1-\l)^{-1/2}$ for every $\l\in[-1,1)$
	and $\phi(1)=0$. 

	Even though $\phi$ is unbounded, it still follows from spectral theory that
	$\int\phi d\SM_{\P}\:=\:\int\left(1-\l\right)^{-1/2}\SM_{\P}(d\l)
	\:=\:\left(\id-\P\right)^{-1/2}$, the inverse operator of $\left(\id-\P\right)^{1/2}$,
	including equality of domains (for a formal argument, see \cite{dunford schwartz}
	Theorem XII.2.9). Thus our problem reduces to showing that
	$\left\{f\in\LTN:v(f,P)<\infty\right\}=\domain\left(\int\phi d\SM_{\P}\right)$.

	Now notice that for every $f\in\LTN$, as $\EM_{f,\P}(\{1\})=0$ by Lemma \ref{E(1)=0},
	$$\int_{[-1,1)}\left(\frac{1}{1-\l}\right)\EM_{f,\P}(d\l)
	=\int_{[-1,1)}\left|\phi(\l)\right|^2\EM_{f,\P}(d\l)=\int\left|\phi\right|^2
	d\EM_{f,\P}.$$
	Thus by Theorem \ref{formula}, for every $f\in\LTN$,
	$v(f,P)=\int_{[-1,1)}\left(\frac{1+\l}{1-\l}\right)\EM_{f,\P}(d\l)$, so
	$$v(f,P)<\infty\qquad\textit{if and only if}\qquad
	\int_{[-1,1)}\left(\frac{1}{1-\l}\right)\EM_{f,\P}(d\l)
	=\int\left|\phi\right|^2d\EM_{f,\P}<\infty.$$
	Thus we would like to show that $\int\left|\phi\right|^2d\EM_{f,\P}$ is finite
	if and only if $f\in\\\domain\left(\int\phi d\SM_{\P}\right)$.

	For every $n\in\N$, let $\phi_n:=\I_{\left(\left|\phi\right|<n\right)}\,\phi$
	and $\D_n:=\phi^{-1}(-n,n)=\phi^{-1}_n(\R)$. Then notice
	$\cup_{k=1}^{\infty}\D_n=\R$ and $\D_n$ is a Borel
	set for every $n$ as $\phi$ is Borel measurable.

	As $\phi$ is positive, notice $\phi_n\leq\phi_{n+1}$ for every $n$. Thus as
	$\phi_n\to\phi$ pointwise for every $\l\in\s(\P)$, by the Monotone Convergence
	Theorem,
	$$\int\left|\phi_n\right|^2d\EM_{f,\P}\to\int\left|\phi\right|^2d\EM_{f,\P}.$$

	As $\phi_n$ is bounded for every $n$, by definition of $\phi_n$ we have
	\begin{align*}
		\int\left|\phi_n\right|^2d\EM_{f,\P}
		&\:=\:\norm{\left(\int\phi_nd\SM_{\P}\right)f}^2\\
		&\:=\:\norm{\left(\int\phi d\SM_{\P}\right)\SM_{\P}\left(
			\cup_{k=1}^n\D_k\right)f}^2\\
		&\:=\:\norm{\SM_{\P}\left(\cup_{k=1}^n\D_k\right)
			\left(\int\phi d\SM_{\P}\right)f}^2.
	\end{align*}
	Thus as $\SM_{\P}(\cup_{k=1}^n\D_k)\to\SM_{\P}(\R)=\id$ as $n\to\infty$ in the
	strong operator topology, we have
	$$\int\left|\phi_n\right|^2d\EM_{f,\P}
	\:=\:\norm{\SM_{\P}\left(\cup_{k=1}^n\D_k\right)\left(\int\phi d\SM_{\P}\right)f}^2
	\:\to\:\norm{\left(\int\phi d\SM_{\P}\right)f}^2.$$
	
	So as $\int\left|\phi_n\right|^2d\EM_{f,\P}\to\int\left|\phi\right|^2d\EM_{f,\P}$
	and $\int\left|\phi_n\right|^2d\EM_{f,\P}\to\norm{\left(\int\phi d\SM_{\P}\right)f}^2$,
	we have
	$$\int\left|\phi\right|^2d\EM_{f,\P}\:=\:\norm{\left(\int\phi d\SM_{\P}\right)f}^2.$$

	Thus as $\norm{\left(\int\phi d\SM_{\P}\right)f}^2<\infty$ \textit{if and only if}
	$f\in\domain\left(\int\phi d\SM_{\P}\right)$, this completes the proof.
\end{proof}
\begin{rem}
	In \cite{kipnis varadhan}, Kipnis and Varadhan state that
	$\{f\in\LTN:v(f,P)<\infty\}=\range\left[\left(\id-\P\right)^{1/2}\right]$. As
	$\range\left[\left(\id-\P\right)^{1/2}\right]=
	\domain\left[\left(\id-\P\right)^{-1/2}\right]$, these are equivalent.
\end{rem}

Now we are ready to prove the ``only if'' direction of Theorem \ref{equiv}, as outlined in
\cite{mira geyer}. Recall the ``only if'' direction of Theorem
\ref{equiv}; if $P$ and $Q$ are $\varphi$-irreducible Markov kernels
reversible with respect to $\pi$, such that $P$ efficiency dominates $Q$, then for every
$f\in\LTN$, $\lft f,\P f\rgt\leq\lft f,\Q f\rgt$.

\bigskip
\noindent
\textit{Proof of ``only if'' Direction of Theorem \ref{equiv}.}

	As $P$ efficiency dominates $Q$, if $f\in\LTN$ such that
	$v(f,Q)<\infty$, then $v(f,P)\leq v(f,Q)<\infty$. Thus
	$\{f\in\LTN:v(f,Q)<\infty\}\subseteq\{f\in\LTN:v(f,P)<\infty\}$.
	So by Lemma \ref{space of fin var}, we have
	$$\domain\left[\left(\id-\Q\right)^{-1/2}\right]\:\subseteq\:
	\domain\left[\left(\id-\P\right)^{-1/2}\right].$$

	By Theorem \ref{formula}, for every $f\in\LTN$,
	$v(f,P)=\int_{[-1,1)}\frac{1+\l}{1-\l}\EM_{f,\P}(d\l)$ and
	$v(f,Q)\\=\int_{[-1,1)}\frac{1+\l}{1-\l}\EM_{f,\Q}(d\l)$. Thus as
	$\int_{[-1,1)}\frac{1+\l}{1-\l}\EM_{f,\P}(d\l)\:=\:
	\norm{f}^2+2\int_{[-1,1)}\frac{\l}{1-\l}\EM_{f,\P}(d\l)$ and
	$\int_{[-1,1)}\frac{1}{1-\l}\EM_{f,\P}(d\l)\:=\:\norm{f}^2
	+\int_{[-1,1)}\frac{\l}{1-\l}\EM_{f,\P}(d\l)$
	and similarly for $Q$, as $P$ efficiency dominates $Q$,
	for every $f\in\LTN$,
	$$\int_{[-1,1)}\frac{1}{1-\l}\EM_{f,\P}(d\l)\leq
	\int_{[-1,1)}\frac{1}{1-\l}\EM_{f,\Q}(d\l).$$

	Furthermore, as seen in the proof of Lemma \ref{space of fin var}, for every\\
	$f\in\domain\left[\left(\id-\P\right)^{-1/2}\right]$, (recall that in the proof
	of Lemma \ref{space of fin var}, $\int\left|\phi\right|^2d\EM_{f,\P}=
	\int_{[-1,1)}\left(\frac{1}{1-\l}\right)\EM_{f,\P}(d\l)$),
	$$\int_{[-1,1)}\left(\frac{1}{1-\l}\right)\EM_{f,\P}(d\l)\:=\:
	\norm{\left(\id-\P\right)^{-1/2}f}^2,$$
	and similarly $\int_{[-1,1)}\left(\frac{1}{1-\l}\right)\EM_{f,\Q}(d\l)=
	\norm{\left(\id-\Q\right)^{-1/2}f}^2$. Thus notice that as\\
	$\domain\left[\left(\id-\Q\right)^{-1/2}\right]\subseteq
	\domain\left[\left(\id-\P\right)^{-1/2}\right]$, for every $f\in\\
	\domain\left[\left(\id-\Q\right)^{-1/2}\right]$,
	$$\norm{\left(\id-\P\right)^{-1/2}f}^2\leq\norm{\left(\id-\Q\right)^{-1/2}f}^2.$$

	So, by Lemma \ref{lemma 6.5}, with $T=\left(\id-\Q\right)$ and
	$N=\left(\id-\P\right)$, for every $f\in\LTN$,
	$\lft f,\left(\id-\Q\right)f\rgt\leq\lft f,\left(\id-\P\right)f\rgt$, and thus
	$$\lft f,\P f\rgt\leq\lft f,\Q f\rgt$$
\qed

The condition $\lft f,\P f\rgt\leq\lft f,\Q f\rgt$ for every $f\in\LTN$ is
equivalent to\\ $\lft f,\left(\Q-\P\right)f\rgt\geq0$ for every $f\in\LTN$, i.e.\
equivalent to $\Q-\P$ being a positive operator. We can
relate this back to the spectrum of the operator $\Q-\P$ with the following lemma.

\begin{lem}\label{pos eigenvalues} If $T$ is a bounded self-adjoint
	linear operator on a Hilbert space $\H$, then $T$ is positive
	if and only if $\s(T)\subseteq[0,\infty)$.
\end{lem}
\begin{proof}
	For the forward direction,
	if $\l<0$, then for every $f\in\H$ such that $f\neq0$,
	by the Cauchy-Schwartz inequality,
	\begin{align*}
		\norm{(T-\l)f}\norm{f}
		&\geq|\lft(T-\l)f,f\rgt|\\
		&=|\lft Tf,f\rgt-\l\norm{f}^2|\\
		&\geq\lft Tf,f\rgt+|\l|\norm{f}^2
		&(\text{as $\l<0$ and by assumption})\\
		&\geq|\l|\norm{f}^2.
		&(\text{by assumption})
	\end{align*}
	Thus as $f\neq0$ and $\l\neq0$,
	$$\norm{(T-\l)f}\geq|\l|\norm{f}>0.$$
	Thus as $T-\l$ is normal (as it is self-adjoint), $T-\l$ is invertible
	(by Lemma \ref{delt}), and $\l\not\in\s(T)$ by definition.

	For the converse, if $\s(T)\subseteq[0,\infty)$, then as $\EM_{f,T}$ is a
	positive measure for every $f\in\H$ (as $\SM_{T}(A)$ is a self-adjoint
	projection for every Borel set $A$),
	$$\lft f,Tf\rgt=\int_{\l\in\s(T)}\l\EM_{f,T}(d\l)=
	\int_{\l\in[0,\infty)}\l\EM_{f,T}(d\l)\geq0,\qquad f\in\H.$$
\end{proof}

This gives us the following.

\begin{thm}\label{Q-P} If $P$ and $Q$ are $\varphi$-irreducible
	Markov kernels reversible with respect to $\pi$,
	then $P$ efficiency dominates $Q$ if and only if
	$\s(\Q-\P)\subseteq[0,\infty)$.
\end{thm}
\begin{proof}
	By Theorem \ref{equiv}, $P$ efficiency-dominates $Q$ \textit{if
	and only if} $\lft f,\P f\rgt\leq\lft f,\Q f\rgt$ for every
	$f\in\LTN$. As $\P$ and $\Q$ are both bounded linear operators,
	this is equivalent to
	$$\lft f,(\Q-\P)f\rgt\geq0,\quad\text{for every $f\in\LTN$,}$$
	or in other words equivalent to $\Q-\P$ being a positive operator.

	Thus as $\Q-\P$ is a bounded self-adjoint linear operator
	on $\LTN$, by Lemma \ref{pos eigenvalues},
	$\Q-\P$ is a positive operator \textit{if and only
	if} $\s(\Q-\P)\subseteq[0,\infty)$.
\end{proof}

As seen in \cite{green han}, antithetic methods can lead to improved efficiency of MCMC
methods. In this paper, we define \textit{antithetic} Markov kernels as Markov kernels
$P$ such that $\s(\P)\subseteq[-1,0]$ when restricted to $\LTN$.
We will show here that antithetic reversible Markov kernels are more efficient
than i.i.d.\ sampling from $\pi$ directly.

\begin{prop}\label{antithetic}
	Let $P$ be a $\varphi$-irreducible Markov kernel reversible
	with respect to $\pi$. Then $P$ is antithetic if and only if $P$ efficiency
	dominates $\Pi$ (the Markov kernel corresponding to i.i.d.\ sampling
	from $\pi$).
\end{prop}
\begin{proof}
	Recall from Section \ref{subsect:mc prelims} that for every
	$f\in\LTN$, for every $x\in\X$,
	$$\Pi f(x)=\int_{y\in\X}f(y)\Pi(x,dy)=\int_{y\in\X}f(y)\pi(dy)=\E_{\pi}(f)=0.$$
	So $\Pi f\equiv0$ for every $f\in\LTN$, and thus $\Pi\equiv0$.

	Say $P$ is antithetic. Then $\s(\Pi-\P)=\s(-\P)\subseteq[0,1]$. Thus
	Theorem \ref{Q-P} shows that $P$ efficiency dominates $Q$.

	On the other hand if $P$ efficiency dominates $\Pi$, then by Theorem
	\ref{Q-P}, $\s(\Pi-\P)\subseteq[0,\infty)$. Thus as $\Pi\equiv0$,
	$\s(\P)\subseteq(-\infty,0]$, and thus $P$ is antithetic.
\end{proof}

We now show that efficiency dominance is partial ordering on the set of
$\varphi$-irreducible reversible Markov kernels reversible with respect
to $\pi$.

\begin{thm}\label{partial ordering} Efficiency dominance is a partial
	order on $\varphi$-irreducible reversible
	Markov kernels, reversible with respect to $\pi$ (reversible with
	respect to the same probability measure).
\end{thm}
\begin{proof}
	Reflexivity is trivial.

	Suppose $P$ and $Q$ are $\varphi$-irreducible reversible
	Markov kernels reversible with respect to $\pi$ such that
	$P$ efficiency dominates $Q$ and $Q$ efficiency dominates $P$.
	Then by Theorem \ref{equiv}, for every $f\in\LTN$,
	$$\lft f,\P f\rgt\leq\lft f,\Q f\rgt\quad\text{and}
	\quad\lft f,\P f\rgt\geq\lft f,\Q f\rgt,$$
	so $\lft f,\P f\rgt=\lft f,\Q f\rgt$ for every $f\in\LTN$.
	Thus $\lft f,(\Q-\P)f\rgt=0$ for every $f\in\LTN$.
	So for every $g,h\in\LTN$, as $\Q$ and $\P$ are
	self-adjoint,
	\begin{align*}
		0
		&=\lft g+h,(\Q-\P)(g+h)\rgt\\
		&=\lft g,(\Q-\P)g\rgt+\lft h,(\Q-\P)h\rgt
		+2\lft g,(\Q-\P)h\rgt\\
		&=2\lft g,(\Q-\P)h\rgt.
	\end{align*}
	So for every $g,h\in\LTN$, $\lft g,(\Q-\P)h\rgt=0$.
	Thus $\Q-\P=0$, so $\P=\Q$, and thus $P=Q$.
	So the relation is antisymmetric.
	
	Suppose $P,Q$ and $R$ are $\varphi$-irreducible reversible
	Markov kernels reversible with respect to $\pi$, such that
	$P$ efficiency dominates $Q$
	and $Q$ efficiency dominates $R$. Then by Theorem
	\ref{equiv}, for every $f\in\LTN$,
	$\lft f,(\RO-\Q)f\rgt\geq0$ and $\lft f,(\Q-\P)f\rgt
	\geq0$. So, for every $f\in\LTN$,
	$$\lft f,(\RO-\P)f\rgt=\lft f,(\RO-\Q)f\rgt+\lft f,
	(\Q-\P)f\rgt\geq0.$$
	And thus by Theorem \ref{equiv}, we have $P$ efficiency dominates
	$R$, so the relation is transitive.
\end{proof}

\section{Combining Chains}\label{sect:mixing}

In this section, we generalise the results of Neal and Rosenthal in \cite{finite}
on the efficiency dominance of combined chains
from finite state spaces to general state spaces. We state the
most general result first, a sufficient condition for the effiency dominance of
combined kernels, and then an important Corollary following
from it.

\begin{thm}\label{mix}
	Let $P_1,\dots,P_l$ and $Q_1,\dots,Q_l$ be Markov kernels
	reversible with respect to $\pi$. Let $\a_1,\dots,\a_l$ be
	mixing probabilities (i.e.\ $\a_k\geq0$ for every $k$, and
	$\sum\a_k=1$) such that $P=\sum\a_kP_k$ and $Q=\sum\a_kQ_k$ are
	$\varphi$-irreducible.

	If $\Q_k-\P_k$ is a positive operator
	(i.e.\ $\s(\Q_k-\P_k)\subseteq[0,\infty)$) for
	every $k$, then $P$ efficiency dominates $Q$.
\end{thm}
\begin{proof}
	By definition of a positive operator (Section \ref{subsect:fa prelims}),
	for every $k\in\{1,\dots,l\}$, we have
	$$\lft f,\left(\Q_k-\P_k\right)f\rgt\geq0,\qquad\forall f\in\LTN.$$
	So, for every $f\in\LTN$,
	$$\lft f,\left(\Q-\P\right)f\rgt=\sum\a_k\lft f,
	\left(\Q_k-\P_k\right)f\rgt\geq0.$$
	Thus as $P$ and $Q$ are also $\varphi$-irreducible
	(and reversible as each $P_k$ and $Q_k$ are reversible),
	by Theorem \ref{equiv},
	$P$ effieciency dominates $Q$.
\end{proof}
\begin{rem} If the Markov kernels $P_1,\cdots,P_l$ and $Q_1,\dots,Q_l$ are
	$\varphi$-irreducible, then Theorem \ref{mix} can be restated as follows.
	Let $P_1,\dots,P_l$ and $Q_1,\dots,Q_l$ be $\varphi$-irreducible
	Markov kernels reversible with respect to $\pi$ and $\a_1,\dots,\a_l$ be
	mixing probabilities. If $P_k$ efficiency dominates $Q_k$ for every $k$,
	then $P=\sum\a_kP_k$ efficiency dominates $Q=\sum\a_kQ_k$.
\end{rem}

The converse of Theorem \ref{mix} is not true, even in the case
where $P_1,\dots,P_l$ and $Q_1,\dots,Q_l$ are $\varphi$-irreducible. For a simple counter
example, take $l=2$, and let $P_1$ and $P_2$ be any $\varphi$-irreducible
Markov kernels, reversible with respect to a probability measure $\pi$ such that
$P_1$ efficiency dominates $P_2$. Then by taking $Q_1=P_2$, $Q_2=P_1$ and $\a_1
=\a_2=1/2$, we have $P=1/2\left(P_1+P_2\right)$ and $Q=1/2\left(Q_1+Q_2\right)=
1/2\left(P_1+P_2\right)$, so $P=Q$. Thus $P$ efficiency dominates $Q$ trivially,
but as $P_1$ efficiency dominates $P_2$, $Q_2=P_1$ efficiency dominates $P_2$,
thus the components do not efficiency dominate each other.

What is true is the following.

\begin{cor}\label{3 comb} Let $P$, $Q$ and $R$ be Markov kernels
	reversible with respect to $\pi$, such that $P$ and $Q$ are
	$\varphi$-irreducible. Then for every $\a\in(0,1)$, $P$ efficiency
	dominates $Q$ if and only if $\a P+(1-\a)R$ efficiency
	dominates $\a Q+(1-\a)R$.
\end{cor}
\begin{proof}
	As $P$ and $Q$ are $\varphi$-irreducible, $\a P+(1-\a)R$
	and $\a Q+(1-\a)R$ are also $\varphi$-irreducible (to see this use
	the same $\s$-finite measure and the fact that for every $x\in\X$ and
	$A\in\F$, $\left(\a P+(1-\a)R\right)^n(x,A)\geq\a^nP^n(x,A)$).

	If $P$ efficiency dominates $Q$, by Theorem \ref{mix},
	$\a P+(1-\a)R$ efficiency dominates $\a Q+(1-\a)R$.

	If $\a P+(1-\a)R$ efficiency dominates $\a Q+(1-\a)R$,
	by Theorem \ref{Q-P},
	$$\s(\Q-\P)=\s(\a^{-1}\left[\a\Q+(1-\a)\RO-(\a\P+(1-\a)\RO)\right])
	\subseteq[0,\infty),$$
	so by Theorem \ref{Q-P} again, $P$ efficiency dominates
	$Q$.
\end{proof}

Thus when swapping only one component, the new combined chain efficiency dominates
the old combined chain if and only if the new component efficiency dominates the old
component.

\section{Peskun Dominance}\label{sect:peskun}

In this section, we show that Theorem \ref{Q-P}, once established, simplifies the
proof that Peskun dominance implies efficiency dominance. Peskun dominance is another
widely used condition, introduced by Peskun in \cite{peskun} for finite state spaces,
and generalized to general state spaces by Tierney in \cite{tierney}. We shall see that
it is a stronger condition than efficiency dominance. We follow the
techniques of Tierney (see \cite{tierney}) to establish a key lemma, then show that
with Theorem \ref{Q-P} already established,
this lemma immediately gives us our result. For a different proof of the fact that
Peskun dominance implies efficiency dominance in the finite state space case,
see \cite{finite}.

We start with our key lemma.

\begin{lem}\label{pesk then pos} If $P$ and $Q$
	are Markov kernels reversible with respect to $\pi$,
	such that $P$ Peskun dominates $Q$, then $\Q-\P$ is a positive
	operator.
\end{lem}
\begin{proof}
	For every $x\in\X$, let $\delta_x:\F\to\{0,1\}$ be the measure such that\\
	$\delta_x(E)=
	\begin{cases}
		1,\quad x\in E\\
		0,\quad o.w.
	\end{cases}$ for every $E\in\F$.

	Then notice that as $P$ and $Q$ are reversible with respect to $\pi$,
	$$\pi(dx)(\delta_x(dy)+P(x,dy)-Q(x,dy))
	=\pi(dy)(\delta_y(dx)+P(y,dx)-Q(y,dx)).$$
	Thus for every $f\in\LTN$, we have
	\begin{align*}
		\lft f,(\Q-\P)f\rgt
		&=\iint_{x,y\in\X}f(x)f(y)
		(Q(x,dy)-P(x,dy))\pi(dx)\\
		&=\int_{x\in\X}f(x)^2\pi(dx)\\
		&\hspace{1.5cm}-\iint_{x,y\in\X}f(x)f(y)
		(\delta_x(dy)+P(x,dy)-Q(x,dy))\pi(dx)\\
		&=\frac{1}{2}\iint_{x,y\in\X}
		\left(f(x)-f(y)\right)^2
		\left(\delta_x(dy)+P(x,dy)-Q(x,dy)\right)\pi(dx).
	\end{align*}

	As $P$ Peskun dominates $Q$, $\left(\delta_x(\cdot)+P(x,\cdot)-Q(x,\cdot)
	\right)$ is a positive measure for $\pi$-almost every $x\in\X$. Thus
	$$\frac{1}{2}\iint_{x,y\in\X}\left(f(x)-f(y)\right)^2
	\left(\delta_x(dy)+P(x,dy)-Q(x,dy)\right)\pi(dx)\geq0.$$
	As $f\in\LTN$ is arbitrary, $\Q-\P$ is a positive operator.
\end{proof}

Now with Theorem \ref{Q-P} we can easily show the following.

\begin{thm}\label{peskun} If $P$ and $Q$ are $\varphi$-irreducible
	Markov kernels reversible with respect to $\pi$, such that $P$
	Peskun dominates $Q$, then $P$ efficiency dominates $Q$.
\end{thm}
\begin{proof}
	By Lemma \ref{pesk then pos}, $\Q-\P$ is a positive operator (i.e.\
	$\s(\Q-\P)\subseteq[0,\infty)$), and thus by Theorem \ref{Q-P},
	$P$ efficiency dominates $Q$.
\end{proof}

The converse of Theorem \ref{peskun} is not true. For a simple example of a kernel that
effieciency dominates but doesn't Peskun dominate another kernel, see Section 7 of
\cite{finite}. Although Peskun dominance can be an easier condition to check, efficiency
dominance is a much more general condition.

\section{Functional Analysis Lemmas}\label{sect:lemma}

We seperate this section into two subsections. In the first subsection, we follow a
parrallel approach to that of Neal and Rosenthal in \cite{finite}
in the finite case, substituting linear algebra for functional analysis where
appropriate, to prove Lemma \ref{lemma 6}. In the second subsection, we follow the
techniques of Mira and Geyer in \cite{mira geyer} to prove Lemma \ref{lemma 6.5}.

\subsection{Proof of Lemma \ref{lemma 6}}\label{subsect:lemma 6}

As shown in \cite{mira geyer}, Lemma \ref{lemma 6} follows from some more
general results in \cite{bendat sherman}. However, these general results
are very technical, and require much more than basic functional analysis to
prove. So we present a different approach using basic functional analysis.
These techniques are similar to what has been done in Chapter V of \cite{bhatia},
as presented by Neal and Rosenthal in \cite{finite}, but generalized for general Hilbert spaces
rather than finite dimensional vector spaces.

We begin with some lemmas about bounded self-adjoint linear operators on a
Hilbert space $\H$.

\begin{lem}\label{add adjoint} If $X,Y,$ and $Z$ are bounded linear 
	operators on a Hilbert space $\H$ such that $\lft f,Xf\rgt
	\leq\lft f, Yf\rgt$
	for every $f\in\H$, and $Z$ is self-adjoint, then
	$\lft f,ZXZf\rgt \leq \lft f,ZYZf\rgt$ for every $f\in\H$.
\end{lem}
\begin{proof} For every $f\in\H$, $Zf\in\H$, so
	$$\lft f,ZXZf\rgt=\lft Zf,XZf\rgt\leq\lft Zf,YZf\rgt=
	\lft f,ZYZf\rgt.$$
\end{proof}

This is where the finite state space case differs from the
general case.
In the finite state space case, $\LTN$ is a finite dimensional vector
space, and thus in order to prove that $\lft f,Tf\rgt\leq\lft f,Nf\rgt$ for
every $f\in\Vect$ if and only if $\lft f,T^{-1}f\rgt\geq\lft f,N^{-1}f\rgt$
for every $f\in\Vect$, when $T$ and $N$ are self-adjoint operators, the only
additional assumption needed is that $T$ and $N$ are \textit{strictly positive},
i.e.\ that for every $f\neq0\in\Vect$, $\lft f,Tf\rgt, \lft f,Nf\rgt>0$.
This is presented by Neal and Rosenthal in \cite{finite}, Section 8.
However, in the general case, as $\LTN$ may not be finite dimensional,
$T$ and $N$ being strictly positive is not a strong enough assumption.
In the general case, it
is possible for $T$ to be strictly positive
and self-adjoint, but not be invertible in the bounded sense. Thus it is possible
that $0\in\s(T)$.
So, we must use a slightly stronger assumption. We must assume that
$\s(T),\s(N)\subseteq(0,\infty)$. In the finite case, this
is equivalent to being strictly positive, however it is stronger in general.

The following lemma is Theorem 12.12 from \cite{rudin}.
We present a more detailed proof below.

\begin{lem}\label{delt}
	If $T$ is a normal bounded linear
	operator on a Hilbert space $\H$, then there exists $\delta
	>0$ such that $\delta\norm{f}\leq\norm{Tf}$ for every $f\in\H$
	if and only if $T$ is invertible.
\end{lem}
\begin{proof}
	For the forward implication, we will show that as $T$ is normal, by the assumption
	it will follow that $T$ is bijective, and then by the assumption once more the inverse
	of $T$ is bounded.

	Firstly, notice that for every $f\in\H$ such that
	$f\neq0$, $\norm{Tf}\geq\delta\norm{f}>0$, so $Tf\neq0$. As $Tf\neq0$ for every
	$f\neq0\in\H$, $T$ is injective.

	As $T$ is normal and injective, $T^*$ is also
	injective, and as $T$ is normal\\ $\range(T)^{\perp}\;=\;\nul(T^*)\;=\;\{0\}$, so
	the range of $T$ is dense in $\H$.

	Now we will show that the range of $T$ is closed,
	and thus $T$ is surjective as the range of $T$ is also dense in $\H$.
	For any $f\in\overline{\range(T)}$, there exists $\{g_n\}_{n\in\N}\subseteq
	\H$ such that $Tg_n\to f$. So for every $m,n\in\N$, by our assumption
	$$\norm{g_n-g_m}=\delta^{-1}\delta\norm{g_n-g_m}\leq\delta^{-1}
	\norm{Tg_n-Tg_m},$$
	so $\{g_n\}\subseteq\H$ is Cauchy as $\{Tg_n\}$ converges. Thus as $\H$ is
	complete (as it is a Hilbert space), there exists $g\in\H$ such that $g_n\to g$.
	As $T$ is bounded, it is also continuous, and thus $Tg_n\to Tg$, and as the limits
	are unique and $Tg_n\to f$ as well, $Tg=f$ and $f\in\range(T)$.
	So $\range(T)$ is closed.

	So as $T$ is bijective, there exists an operator $T^{-1}$ such that $TT^{-1}f=f$
	for every $f\in\H$. By our assumption, letting $C=\delta^{-1}$, we have
	$$C\norm{f}=C\norm{TT^{-1}f}\geq\norm{T^{-1}f}$$
	for every $f\in\H$, and so $T^{-1}$ is bounded.

	For the converse, say $T$ is invertible. Then let $\delta=
	\norm{T^{-1}}^{-1}$. Then for every $f\in\H$, by definition of $\delta$,
	$$\delta\norm{f}=\delta\norm{T^{-1}Tf}\leq\delta\norm{T^{-1}}\norm{Tf}=\norm{Tf}.$$
\end{proof}
\begin{rem}
	The assumption that $T$ be normal in the preceding lemma is only to show us that
	$T$ is bijective in the ``only if'' direction. In general, if $T$ is a bounded linear
	operator, not necessarily normal, if $T$ is bijective and there exists $\delta>0$
	such that $\norm{Tf}\geq\delta\norm{f}$ for every $f\in\H$, then $T$ is invertible.
	Furthermore, the ``if'' direction of Lemma \ref{delt} does not require that $T$
	be normal.
\end{rem}

And now we can prove Lemma \ref{lemma 6}.

\bigskip
\noindent
\textit{Proof of Lemma \ref{lemma 6}.}
	Say $\lft f,Tf\rgt\leq\lft f,Nf\rgt$ for every $f\in\H$.
	
	As $\s(N)\subseteq(0,\infty)$, $N$ is invertible, and $N^{-1/2}$
	is a well defined bounded self-adjoint linear operator.
	Similarly, $T^{1/2}$ is also a well defined bounded self-adjoint
	linear operator.
	
	So, for every $f\in\H$, we have
	$$\lft f,N^{-1/2}TN^{-1/2}f\rgt=\lft T^{1/2}N^{-1/2}f,T^{1/2}N^{-1/2}f\rgt
	=\norm{T^{1/2}N^{-1/2}f}^2\geq0.$$

	Furthermore, as $\s(T)\subseteq(0,\infty)$, $T$ is invertible,
	so by Lemma \ref{delt}, there exists $\de_T>0$ such that
	$\norm{Tf}\geq\de_T\norm{f}$ for every $f\in\H$.
	Also, notice that $\s(N^{-1/2})\subseteq(0,\infty)$,
	thus by Lemma \ref{delt},
	there exists $\de_1>0$ such that
	$\norm{N^{-1/2}f}\geq\de_1\norm{f}$ for every $f\in\H$.
	So, for every $f\in\H$,
	$$\norm{N^{-1/2}TN^{-1/2}f}\geq
	\de_1\norm{TN^{-1/2}f}
	\geq\de_1\de_T\norm{N^{-1/2}f}
	\geq\de_1\de_T\de_1\norm{f},$$
	so by Lemma \ref{delt},
	$N^{-1/2}TN^{-1/2}$ is invertible,
	and thus $0\not\in\s(N^{-1/2}TN^{-1/2})$.

	By using Lemma \ref{add adjoint} with $X=T$, $Y=N$ and
	$Z=N^{-1/2}$, for every $f\in\H$,
	$$\lft f,N^{-1/2}TN^{-1/2}f\rgt\leq\lft f,N^{-1/2}NN^{-1/2}f\rgt
	=\norm{f}^2.$$
	So if $\l>1$, for any $f\in\H$, as $0\leq
	\lft N^{-1/2}TN^{-1/2}f,f\rgt\leq\norm{f}^2$, by the
	Cauchy-Schwartz inequality,
	$\norm{(N^{-1/2}TN^{-1/2}-\l)f}\geq|1-\l|\norm{f}$,
	and as $|1-\l|>0$, by Lemma \ref{delt},
	$(N^{-1/2}TN^{-1/2}-\l)$ is invertible, so
	$\l\not\in\s(N^{-1/2}TN^{-1/2})$.

	Thus we have $\s(N^{-1/2}TN^{-1/2})\subseteq(0,1]$.

	Let $K$ denote the inverse of $N^{-1/2}TN^{-1/2}$, i.e.\ let
	$K=(N^{-1/2}TN^{-1/2})^{-1}$. Furthermore, we have
	$\s(K)\subseteq[1,\infty)$. So for every $f\in\H$,
	$\norm{f}^2\leq\lft f,Kf\rgt$.

	So by using Lemma \ref{add adjoint}, with $X=\id$,
	$Y=K$ and $Z=N^{-1/2}$, for every $f\in\H$,
	\begin{align*}
		\lft f,N^{-1}f\rgt
		&=\lft f,N^{-1/2}\id N^{-1/2}f\rgt\\
		&\leq\lft f,N^{-1/2}K N^{-1/2}f\rgt\\
		&=\lft f,N^{-1/2}(N^{-1/2}TN^{-1/2})^{-1}
		N^{-1/2}f\rgt\\
		&=\lft f,N^{-1/2}N^{1/2}T^{-1}N^{1/2}
		N^{-1/2}f\rgt\\
		&=\lft f,T^{-1}f\rgt.
	\end{align*}

	For the other direction, replace $N$ with $T^{-1}$
	and $T$ with $N^{-1}$.
\qed

\subsection{Proof of Lemma \ref{lemma 6.5}}\label{subsect:lemma 6.5}

Here we follow the same steps of \cite{mira geyer} to prove Lemma \ref{lemma 6.5}.

\begin{lem}\label{sub}
	If $T$ is a self-adjoint, injective, positive and bounded operator on
	the Hilbert space $\H$, then $\domain(T^{-1})\subseteq\domain(T^{-1/2})$.
\end{lem}
\begin{proof}
	Let $f\in\domain(T^{-1})=\range(T)$. Then there exists $g\in\H$
	such that $Tg=f$. So, as $T$ is positive, $T^{1/2}$ is
	well-defined, so $T^{1/2}g=h\in\H$. Thus notice
	$T^{1/2}h=T^{1/2}T^{1/2}g=Tg=f$, so $f\in\range(T^{1/2})
	=\domain(T^{-1/2})$.
\end{proof}

The next lemma is a generalization of Lemma 3.1 of \cite{mira geyer} from
real Hilbert spaces to possibly complex ones. This generalization is simple
but unnecessary for us, as we are dealing with real Hilbert spaces anyways.

\begin{lem}\label{gen form}
	If $T$ is a self-adjoint, injective, positive and bounded linear
	operator on the Hilbert space $\H$,
	then for every $f\in\H$,
	$$\lft f,Tf\rgt=\sup_{g\in\domain(T^{-1/2})}
	\left[\lft g,f\rgt+\lft f,g\rgt-\lft T^{-1.2}g,T^{-1/2}g\rgt
	\right].$$
\end{lem}
\begin{proof}
	As $T$ is injective and self-adjoint, the inverse of $T$,
	$T^{-1}:\range(T)\to\H$, is densely defined and self-adjoint
	(see Proposition X.2.4 (b) of \cite{conway}).

	For every $f\in\range(T)=\domain(T^{-1})$, there exists
	$g\in\H$ such that $Tg=f$. Thus as $T$ is positive and
	self-adjoint,
	$$\lft f,T^{-1}f\rgt=\lft Tg,g\rgt=\lft g,Tg\rgt\geq0,$$
	so $T^{-1}$ is also positive. In particular, this means
	that $T^{1/2}$ and $T^{-1/2}$ are well-defined.

	By Lemma \ref{sub}, $\domain(T^{-1})\subseteq\domain(T^{-1/2})$.
	So, let $f\in\H$. Let $h=Tf$. Then for every
	$g\in\domain(T^{-1/2})= \range(T^{1/2})$,
	\begin{align*}
		\lft f,&Tf\rgt-\left(\lft g,f\rgt+\lft f,g\rgt
			-\lft T^{-1/2}g,T^{-1/2}g\rgt\right)\\
		&=\lft T^{1/2}f,T^{1/2}f\rgt-\lft g,T^{-1}h\rgt
			-\lft T^{-1}h,g\rgt
			+\lft T^{-1/2}g,T^{-1/2}g\rgt\\
		&=\lft T^{-1/2}h,T^{-1/2}h\rgt-\lft T^{-1/2}g,
			T^{-1/2}h\rgt-\lft T^{-1/2}h,T^{-1/2}g\rgt
			+\lft T^{-1/2}g,T^{-1/2}g\rgt\\
		&=\lft T^{-1/2}(h-g),T^{-1/2}(h-g)\rgt\\
		&=\norm{T^{-1/2}(h-g)}^2\\
		&\geq0.
	\end{align*}

	As $h\in\domain(T^{-1})$ and $\domain(T^{-1})\subseteq
	\domain(T^{-1/2})$, $h\in\domain(T^{-1/2})$. So, as $T$ is
	self-adjoint,
	$$\lft h,f\rgt+\lft f,h\rgt-\lft T^{-1/2}h, T^{-1/2}h\rgt
	=\lft Tf,f\rgt+\lft f,Tf\rgt-\lft Tf, T^{-1}Tf\rgt
	=\lft f,Tf\rgt.$$
\end{proof}

With Lemma \ref{gen form} established, the proof of Lemma \ref{lemma 6.5} is straightforward.

\bigskip
\noindent
\textit{Proof of Lemma \ref{lemma 6.5}.}
	Let $f\in\H$. Then by Lemma \ref{gen form},
	\begin{align*}
		\lft f,Tf\rgt
		&=\sup_{g\in\domain(T^{-1/2})}\lft g,f\rgt+\lft f,g\rgt
		-\lft T^{-1/2}g,T^{-1/2}g\rgt\\
		&\leq\sup_{g\in\domain(N^{-1/2})}\lft g,f\rgt+
		\lft f,g\rgt -\lft N^{-1/2}g,N^{-1/2}g\rgt\\
		&=\lft f,Nf\rgt.
	\end{align*}
\qed

\bigskip
\noindent
\textbf{Acknowledgements.} We thank Austin Brown and Heydar Radjavi for some very helpful
discussions about the Spectral Theorem which allowed us to prove Lemma \ref{space of fin var}.
This work was partially funded by NSERC of Canada.

\end{document}